\documentclass[11pt]{article}
\usepackage{amsthm,amsmath,amssymb}
\usepackage{natbib}
\usepackage{multirow}
\usepackage[pdftex]{graphicx}
\usepackage{subfigure}
\usepackage{makecell}
\usepackage{booktabs}
\usepackage{array}
\usepackage{fullpage}
\usepackage{url}
\usepackage{algorithm}
\usepackage{algorithmic}
\usepackage{bm}
\usepackage{main}
\usepackage{mathtools}
\usepackage{wrapfig}
\usepackage{lipsum}
\usepackage{mathrsfs}
\usepackage{dsfont}
\usepackage{esint}

\usepackage{natbib}
\usepackage{multirow}

\usepackage{appendix}

\newcolumntype {Q}{>{$\displaystyle}l<{$}}
\newcolumntype {A}{>{$}c <{$}}

\def\tr{\mathop{\text{tr}}\kern.2ex}

\def\N{{\mathbb N}}
\def\Z{{\mathbb Z}}

\def\cov{{\rm Cov}}
\def\var{{\rm Var}}

\long\def\comment#1{}

\def\tr{\mathop{\text{Tr}}}

\providecommand{\norm}[1]{\vvvert#1\vvvert}

\newcommand{\bel}{\begin{eqnarray}\label}
\newcommand{\eel}{\end{eqnarray}}
\newcommand{\bes}{\begin{eqnarray*}}
\newcommand{\ees}{\end{eqnarray*}}

\def\reals{{\mathbb{R}}}

\usepackage[usenames,dvipsnames,svgnames,table]{xcolor}
\usepackage[colorlinks=true,
            linkcolor=blue,
            urlcolor=blue,
            citecolor=blue]{hyperref}
            
\numberwithin{equation}{section}
\numberwithin{theorem}{section}
\numberwithin{corollary}{section}
\numberwithin{asmp}{section}
\numberwithin{definition}{section}  

\begin{document}

\setlength{\abovedisplayskip}{5pt}
\setlength{\belowdisplayskip}{5pt}
\setlength{\abovedisplayshortskip}{5pt}
\setlength{\belowdisplayshortskip}{5pt}

\title{\LARGE A Note on Mixing in High Dimensional Time Series}

\author{Jiaqi Yin \thanks {Department of Biostatistics, University of Washington, Seattle, WA 98195, USA.}}



\date{}

\maketitle

\begin{abstract}
Various mixing conditions have been imposed on high dimensional time series, including the strong mixing ($\alpha$-mixing),   maximal  correlation  coefficient ($\rho$-mixing), absolute regularity ($\beta$-mixing), and $\phi$-mixing. $\alpha$-mixing condition is a routine assumption when studying autoregression models. $\rho$-mixing can lead to $\alpha$-mixing. 
In this paper, we prove a way to verify $\rho$-mixing under a high-dimensional triangular array time series setting by using the Pearson's $\phi^2$, mean square contingency. Vector autoregression model VAR(1) and vector autoregression moving average VARMA(1,1) are proved satisfying $\rho$-mixing condition based on low rank setting. 
\end{abstract}


\section{Introduction}
Dependent sequence data is normal in time series. 
Several coefficients are studied broadly to measure the dependence of two $\sigma$-fields. In the paper, we mainly talk about $\rho$-mixing. The strong mixing coefficient ($\alpha$-mixing coefficient)  is first introduced by  \cite{Rosenblatt1956Mixing}. The strong mixing condition is highly used when studying autoregression models, \cite{Andrews1991, Liebscher2005}. \cite{Rozano1960rho} defined maximal correlation coefficient ($\rho$-mixing coefficient), and also showed that $\alpha(\mathcal{A},\mathcal{B}) \leq \rho(\mathcal{A},\mathcal{B})\leq 2\pi\rho(\mathcal{A},\mathcal{B})$, where $\mathcal{A}$ and $\mathcal{B}$ are two $\sigma$-fields. Therefore, $\rho$-mixing implies $\alpha$-mixing. Absolute regularity ($\beta$-mixing coefficient) was original introduced by \cite{Volkonskii1959beta}, and several equivalent forms of $\beta$-mixing is showed in \cite{BradleyMixing}. 

The paper will be organized as two sections. In Section 1, we will introduce related definitions and the notations we use.  In Section 2, we show that under certain weak assumptions, stationary Markov Chain generated by vector autoregression model (VAR) or vector autoregression moving average model (VARMA) satisfies $\rho$-mixing.  The whole setting is based on   the triangle array setting.

\begin{definition}\label{def:rho_n}
\cite{BradleyMixing,Beare2010Copulas}. The $\rho$-mixing coefficients $\left\{ \rho(n) :n\in\mathbb{N}\right\} $
that correspond to the sequence of random variables $\boldsymbol{X}:=\left\{ X(t) \right\} _{t\in\mathbb{Z}}$
on $(\Omega,\mathcal{F},\mathbb{P}) $ is defined as
\[
\rho(n) =\rho(\boldsymbol{X},n) :=\sup_{J\in\mathbb{Z}}\sup_{f,g}\left|Corr(f,g) \right|,
\]
where the second supremum is taken over $f\in\mathcal{L}_{\text{real}}^{2}(\mathcal{F}_{-\infty}^{J}) $,
$g\in\mathcal{L}_{\text{real}}^{2}(\mathcal{F}_{J+n}^{\infty}) $,
and $\mathcal{F}_{i}^{j}:=\sigma(X(t) ,i\leq t\leq j) $
represents the $\sigma$-fields generated by $\left\{ X(t) \right\} _{i\leq t\leq j}$.
\end{definition}

Copula $C$ is used to characterize the dependence between random variables. \cite{Beare2010Copulas} defined the maximal correlation $\rho_C$ of the copula $C$, and showed that $\rho_C<1$ can lead to $\rho$-mixing. The followings review the definitions of copula and $\rho_C$.
\begin{definition} \label{def:copula}
\cite{NelsenCopulas}. $C:\left[0,1\right]^{d}\rightarrow\left[0,1\right]$ is a $d$-dimensional
copula if $C$ is a joint cumulative distribution function of a $d$-dimensional
random vector on the unit cube $[0,1]^{d}$ with uniform marginals. 
\end{definition}

\begin{definition} \label{def:rhoC}
\cite{Beare2010Copulas}. The maximal correlation $\rho_{C}$ of the copula $C$ is given by
\[
\sup_{f,g}\left|\int_{\left[0,1\right]^{d_{1}}}\int_{\left[0,1\right]^{d_{2}}}f(\boldsymbol{x}) g(\boldsymbol{y}) C(d\boldsymbol{x},d\boldsymbol{y}) \right|,
\]
where the supremum is taken over all $f\in L_{d_{1}}\left[0,1\right],g\in L_{d_{2}}\left[0,1\right]$
such that $\int f=\int g=0$ and $\int f^{2}=\int g^{2}=1$. 
\end{definition}

\subsection{Triangular Array Setting}
Triangular array setting is well-accepted to model high-dimensional data. \cite{han2019probability} apply this setting to time series models.
$\left\{ \boldsymbol{X}_{p_{T}} (t)\right\} _{t\in\mathbb{Z}}$
is a $p_{T}-$dimensional multivariate time series, where $T\in\mathbb{N}^{+}$,
$p_{T}=p(T) \in\mathbb{N}^{+}$, $\boldsymbol{X}_{p_{T}} (t):=(X_{1} (t),X_{2} (t),\ldots,X_{p_{T}} (t)) $.
For any $T$, we only observe a length $T$ fragment, $\left\{\boldsymbol{X}_{p_T}(t),t\in [T]\right\}$
of the time series $\left\{ \boldsymbol{X}_{p_{T}} (t)\right\} _{t\in\mathbb{Z}}$.
To be more detailed, consider Triangular Array Setting,
\begin{align*}
T & =1\mathbb\quad\text{observe }\boldsymbol{X}_{p_{1}}(1);\\
T & =2\mathbb\quad\text{observe }\boldsymbol{X}_{p_{2}}(1),\boldsymbol{X}_{p_{2}}(2);\\
T & =3\mathbb\quad\text{observe }\boldsymbol{X}_{p_{3}}(1),\boldsymbol{X}_{p_{3}}(2),\boldsymbol{X}_{p_{3}}(3);\\
\vdots & \mathbb\quad\mathbb\quad\mathbb\quad\mathbb\quad\vdots\\
T & =T\mathbb\quad\text{observe }\boldsymbol{X}_{p_{T}} (1),\boldsymbol{X}_{p_{T}} (2),\boldsymbol{X}_{p_{T}} (3),\ldots,\boldsymbol{X}_{p_{T}} (T);\\
\vdots & \mathbb\quad\mathbb\quad\mathbb\quad\mathbb\quad\vdots
\end{align*}
Random vectors in each row can be generated from a same model, e.g.
VAR, and AVRMA models in our example. Random vectors in different rows
are allowed to be from different generating models. 

\subsection{Notation}
Let $\N$,       $\N^+$, $\Z$, and $\reals$ represent the sets of natural numbers, non-zero natural numbers, integers, and real numbers. For each $n \in \mathbb{N}$, we define $[n]=\left\{1,2,\ldots,n\right\}$.
$\boldsymbol{X}_p\in \mathbb{R}^p$ is a $p$-dimensional random vector, $\boldsymbol{X}_p:=\left(X_i,i\in [p]\right)$.
$I_k$ is the $k\times k$ identical matrix.
For a given matrix $Q \in\reals^{n\times n}$, $Q^{\intercal   }$ is the transpose of $Q$. $\sigma_i(Q)$ represents the $i$th largest singular value of $Q$, and it can be defined as $\sigma_i(Q)  = \sqrt{\lambda_i(Q Q ^\intercal   )}$, where $\lambda_i(\cdot)$ is the $i$th largest eigenvalue of $Q Q ^\intercal   $. If $Q$ is a symmetric matrix, there is $\sigma_i(Q) = |\lambda_i(Q)|$. $\lambda_{\max}(\cdot)$ and $\lambda_{\min}(\cdot)$ are the maximum and minimum eigenvalues respectively. Denote $\norm{\cdot}$ as the operator norm, or spectral norm, and it also can be defined as the largest singular value,  $\norm{Q}:=\sigma_1(Q)$. $A\succ0$ means $A$ is positive definite, and $A\succeq0$ means $A$ is positive semi-definite. For random vector $\boldsymbol{X}$, denote its covariance-variance matrix as $\Gamma_X$.
Throughout the paper, let $M,M^{'}>0$ be two generic absolute constants, whose actual values may vary at different locations.

\section{VAR and VARMA Models}

Before providing the main results derived in this section, we first introduce  properties of $\rho$-coefficients applied in high dimensional Markov Chain. 

\begin{lemma}\label{lem:rho_decay}
Suppose $\boldsymbol{X}_{p_{T}} =\left\{ \boldsymbol{X}_{p_{T}} (t)\right\} _{t\in\mathbb{Z}}$,
a $p_{T}$-dimensional multivariate time series, is also a (not
necessary stationary) Markov Chain. If for any $T\in\mathbb{N}^{+}$,
there exits some integer $N$ independent of $T$, such that $\rho(\boldsymbol{X}_{p_{T}} ,N) <\kappa<1$,
where $\kappa>0$ is an absolute constant independent of $T$, then
for any $T\in\mathbb{N}^{+}$, there exists
\[
\rho(\boldsymbol{X}_{p_{T}} ,n) \le Ae^{-\gamma n},
\]
 where $\gamma>0$ and $A<\infty$ are two absolute constants not depending
on $n$ or $T$.
\end{lemma}

\begin{proof}
 For each positive integer $T$,  integer $j$, and  positive integer $m$, by the definition of $\rho$-mixing coefficient and the property of Markov Chain \ref{A:prop:markov rho}, we have 
 \[
\rho\left[\sigma\{\boldsymbol{X}_{p_{T}} (t),t\leq j\} ,\sigma\{\boldsymbol{X}_{p_{T}} (t),t\geq j+(m+1) N\} \right] =\rho[\sigma\{\boldsymbol{X}_{p_{T}} (j)\} ,\sigma\{\boldsymbol{X}_{p_{T}} (j+(m+1) N)\}].
\]
Using the properties of $\rho$-mixing coefficient in Markov Chain again \ref{A:prop:markov rho}, we further have the inequality, 
\begin{align*}
\rho[\sigma\{\boldsymbol{X}_{p_{T}} (j)\} ,\sigma\{\boldsymbol{X}_{p_{T}} (j+(m+1) N)\}] & \leq\rho[\sigma\{\boldsymbol{X}_{p_{T}} (j)\} ,\sigma\{\boldsymbol{X}_{p_{T}} (j+mN)\}] \\
 & \mathbb\quad\cdot\rho[\sigma\{\boldsymbol{X}_{p_{T}} (j+mN)) ,\sigma(\boldsymbol{X}_{p_{T}} (j+(m+1) N)\}].
\end{align*}
By Definition \ref{def:rho_n} , we deduce 
\begin{align*}
    \rho[\sigma\{\boldsymbol{X}_{p_{T}} (j)\} ,\sigma\{\boldsymbol{X}_{p_{T}} (j+mN)\}] \cdot\rho[\sigma\{\boldsymbol{X}_{p_{T}} (j+mN)) ,\sigma(\boldsymbol{X}_{p_{T}} (j+(m+1) N)\}]
    & \leq\rho(\boldsymbol{X}_{p_{T}} ,mN) \\
    & \mathbb\quad \cdot \rho(\boldsymbol{X}_{p_{T}} ,N) .
\end{align*}

Since $j$ is random in above inequality, we take the supremum of the right hand side,  with the condition $\rho(\boldsymbol{X}_{p_{T}} ,N) <\kappa$, we will have
\[
\rho(\boldsymbol{X}_{p_{T}} ,mN) \cdot\rho(\boldsymbol{X}_{p_{T}} ,N) \leq\rho(\boldsymbol{X}_{p_{T}} ,mN) \kappa.
\]

Hence by induction, for any $T\in\mathbb{N}^{+}$, $\rho(\boldsymbol{X}_{p_{T}} ,mN) <\kappa^{m}$.
Our proof is completed by \ref{A:prop_expodecay}.
\end{proof}

Verifying whether some models satisfy $\rho$-mixing can be a challenge, since it is hard to find $f$ and $g$ from  two infinite classes. However,  \cite{Remillard2012CopulaTimeSeries} extended the work of \cite{Lancaster1958PhiSq} and showed that for random vector, bounded square contingency ($\phi_{C}^2$) would yield $\rho_C<1$. 

\begin{definition}\label{def:phiSq}
 Suppose random vector $(\boldsymbol{X},\boldsymbol{Y}) =(X_{1},\ldots,X_{d_{1}},Y_{1},\ldots,Y_{d_{2}}) \in\mathbb{R}^{d_{1}+d_{2}}$
has cumulative distribution function $H(\boldsymbol{x},\boldsymbol{y}) $
and density $h(\boldsymbol{x},\boldsymbol{y}) $. The marginal
distributions of $\boldsymbol{X}=(X_{1},\ldots,X_{d_{1}}) \in\mathbb{R}^{d_{1}}$
and $\boldsymbol{Y}=(Y_{1},\ldots,Y_{d_{2}}) \in\mathbb{R}^{d_{2}}$
are $P(\boldsymbol{x}) $ and $Q(\boldsymbol{y}) $
with density $p(\boldsymbol{x}) $ and $q(\boldsymbol{y}) $
respectively. Then the Pearson's $\phi^{2}$ is 
\[
\phi^{2}=-1+\int\int\frac{h^{2}(\boldsymbol{x},\boldsymbol{y}) }{p(\boldsymbol{x}) q(\boldsymbol{y}) }d\boldsymbol{x}d\boldsymbol{y}.
\]
\end{definition}
Suppose $\boldsymbol{X}=(X_{1},\ldots,X_{d_{1}}) $ has continuous
marginal distribution $F_{1}(X_1),\ldots,F_{d_{1}}(X_{d_1})$, and $\boldsymbol{Y}=(Y_{1},\ldots,Y_{d_{2}}) $
has continuous marginal distribution $G_{1}(Y_1),\ldots,G_{d_{2}}(Y_{d_2})$. According
to Sklar's Theorem \cite{SklarTheom}, we know there exists a unique
$(d_{1}+d_{2}) $-dimensional random copula $C$ such that
for all $\boldsymbol{x}=(x_{1},\ldots,x_{d_{1}})$, $\boldsymbol{y}=(y_{1},\ldots,y_{d_{2}})$, 
\begin{equation}
H(\boldsymbol{x},\boldsymbol{y}) =C(F_{1}(x_{1}) ,\ldots,F_{d_{1}}(x_{d_{1}}) ,G_{1}(y_{1}) ,\ldots,G_{d_{2}}(y_{d_{2}})). \label{eq:H(x,y)}
\end{equation}
 Assuming that the copula $C$ is absolutely continuous with density
$c$, and the densities $f_{i}$ of $F_{i}$ and $g_{j}$ of $G_{j}$
exist for all $i\in\left\{ 1,\ldots,d_{1}\right\} $ and $j\in\left\{ 1,\ldots,d_{2}\right\} $,
then the joint density of $H$ is 
\[
h(\boldsymbol{x},\boldsymbol{y}) =c(F_{1}(x_{1}) ,\ldots,F_{d_{1}}(x_{d_{1}}) ,G_{1}(y_{1}) ,\ldots,G_{d_{2}}(y_{d_{2}}) ) \prod_{i=1}^{d_{1}}f_{i}(x_{i}) \prod_{j=1}^{d_{2}}g_{j}(y_{j}) .\]
Define $\boldsymbol{U}=(F_{1}(X_{1}) ,\ldots,F_{d_{1}}(X_{d_{1}}))$, and
$\boldsymbol{V}=(G_{1}(Y_{1}) ,\ldots,G_{d_{2}}(Y_{d_{2}}) ) $.
Then the marginal distribution function of $\boldsymbol{X}$ and $\boldsymbol{Y}$
can be derived from (\ref{eq:H(x,y)}) as $P(\boldsymbol{x}) =C(\boldsymbol{u},\boldsymbol{1}) $
, $Q(\boldsymbol{y}) =C(\boldsymbol{1},\boldsymbol{v}) $
with density $p(\boldsymbol{x}) = c(\boldsymbol{u},1) \prod_{i=1}^{d_{1}}f_{i}(x_{i}) $,
$q(\boldsymbol{y}) =c(1,\boldsymbol{v}) \prod_{j=1}^{d_{2}}g_{j}(y_{j}) $
respectively. Then define $\phi_{C}^{2}=-1+\int\int\frac{c(\boldsymbol{u},\boldsymbol{v}) }{c(\boldsymbol{u},1) c(1,\boldsymbol{v})}d\boldsymbol{u}d\boldsymbol{v}$,
and obviouly,
\begin{equation}\label{phic=phi}
\phi_{C}^{2}=-1+\int\int\frac{h^{2}(\boldsymbol{x},\boldsymbol{y}) }{p(\boldsymbol{x}) q(\boldsymbol{y}) }d\boldsymbol{x}d\boldsymbol{y}=\phi^{2}.
\end{equation}

With different $T$, different dimensional multivariate time series
will be observed. We then add $p_T$ as the index of the
following functions to tell the difference. 

Consider stationary Markov Chain $\left\{ \boldsymbol{X}_{p_{T}}  (t )\right\} _{t\in\mathbb{Z}}$.
Let  $F_{p_T}(\cdot)$ be the cumulative distribution function of $\boldsymbol{X}_{p_{T}}  (t )$,
and  $H_{2p_T}(\cdot)$ be the joint cumulative distribution function of $ (\boldsymbol{X}_{p_{T}}  (t-1 ),\boldsymbol{X}_{p_{T}}  (t ) )$.
Let $C_{2p_T} $ be the copula associated with the $2p_{T}-$dimensional
random vector $ (\boldsymbol{X}_{p_{T}}  (t-1 ),\boldsymbol{X}_{p_{T}}  (t ) )$.
The copula  $Q_{p_T}$ of $\boldsymbol{X}_{p_{T}}  (t-1 )$
is the same as the copula of $\boldsymbol{X}_{p_{T}}  (t)$,  $Q_{p_T} (\boldsymbol{u} )=C_{2p_T}  (\boldsymbol{u},\mathbf{1} )=C_{2p_T}  (\boldsymbol{1},\boldsymbol{u} )$, any $\boldsymbol{u}\in\left[0,1\right]^{p_{T}}$.
Define $\boldsymbol{U}_{p_{T}}  (t )=F  \{\boldsymbol{X}_{p_{T}}  (t ) \}$.
Then $\left\{ \boldsymbol{U}_{p_{T}}  (t )\right\} _{t\in\mathbb{Z}}$
is a $p_{T}-$dimensional time series such that the cumulative distribution
function of $ (\boldsymbol{U}_{p_{T}}  (t-1 ),\boldsymbol{U}_{p_{T}}  (t ) )$
is $C_{2p_T}$, denoting as $ (\boldsymbol{U}_{p_{T}}  (t-1 ),\boldsymbol{U}_{p_{T}}  (t ) )\sim C_{2p_T} $,
and the marginal distribution $\boldsymbol{U}_{p_T}  (t )\sim Q_{p_T} $.
If we consider the $\rho$-coefficient of $\sigma \{\boldsymbol{U}_{p_{T}}  (t-1 ) \}$
and $\sigma \{\boldsymbol{U}_{p_{T}}  (t ) \}$,
then we can set $\rho_{C}$ as $\rho_{C}=\rho (\boldsymbol{U}_{p_{T}} ,1 )$ \citep {Remillard2012CopulaTimeSeries},
specifically, $\rho_{C_{2p_T}}$.

\subsection{Main theorem}
\begin{theorem}\label{thm:rho-mixing}
Suppose $\left\{ \boldsymbol{X}_{p_{T}} (t)\right\} _{t\in\mathbb{Z}}$
is a $p_{T}-$dimensional multivariate stationary Markov Chain. Assume
  $F_{p_T}$ and  $H_{2p_T}$ are the marginal distribution of $\boldsymbol{X}_{p_{T}} (t)$ and $(\boldsymbol{X}_{p_{T}} (t-1), \boldsymbol{X}_{p_{T}} (t)) $ respectively, and their densities $f_{p_T} >0$, $h_{2p_T} >0$ are almost surely on $\mathbb{R}^{p_{T}}$ and $\mathbb{R}^{2p_{T}}$. If there exists an absolute constant $\gamma>0$ not
depending on $p_{T}$, such that
\[
\lim_{p_{T}\rightarrow\infty}\phi^{2}(p_{T}) =\lim_{p_{T}\rightarrow\infty}-1+\int_{\mathbb{R}^{p_{T}}}\int_{\mathbb{R}^{p_{T}}}\frac{\left\{ h_{2p_T} (\boldsymbol{x},\boldsymbol{y}) \right\} ^{2}}{f_{p_T}(\boldsymbol{x}) f_{p_T} (\boldsymbol{y}) }d\boldsymbol{x}d\boldsymbol{y}<\gamma,
\]
then for any $T\in\mathbb{N}^{+}$, any $n\in\mathbb{N}^{+}$,
there is 
\[
\rho(\boldsymbol{X}_{p_{T}} ,n) \leq M\delta^{n},
\]
where $M,\delta$ do not depend on $T$ or $n$. 
\end{theorem}

\begin{remark}
For fixed $T$, Theorem $\ref{thm:rho-mixing}$ will lead to $\rho_{C_{2p_T}}<1$ by (\ref{phic=phi}) and Proposition 2 in \cite{Remillard2012CopulaTimeSeries}. Then $\rho(\boldsymbol{X}_{p_{T}} ,n) \leq M\delta^n$ will be yielded according to the Theorem 4.1 in \cite{Beare2010Copulas}.
\end{remark}
\begin{proof}[Proof of Theorem \ref{thm:rho-mixing}]
Suppose $C_{2p_T}$ is the copula associated with $2p_{T}$-dimensional
random vector $\left(\boldsymbol{X}_{p_{T}}(t-1),\boldsymbol{X}_{p_{T}}(t)\right)$,
and $Q_{p_T}$ is the copula of $\boldsymbol{X}_{p_{T}}\left(t\right)$.
Note that $\left\{ \boldsymbol{X}_{p_{T}}\left(t\right)\right\} _{t\in\mathbb{Z}}$
is a stationary Markov chain, then $C_{2p_T}$ and $Q_{p_T}$ will remain
the same for any $t\in\mathbb{Z}$. If we can prove for any $T$, $\rho\left(\boldsymbol{X}_{p_{T}},1\right)<\kappa<1$,
 then by Lemma \ref{lem:rho_decay} we will know that $\rho\left(\boldsymbol{X}_{p_{T}},n\right)\leq M\delta^{n}$,
where absolute constants $\kappa, M,\delta$ do not depend on $T$ or $n$.  From  (\ref{phic=phi}), we know $\phi_{C}^{2}\left(p_{T}\right)=\phi^{2}\left(p_{T}\right)$, and as showed in \cite{Beare2010Copulas},
$\rho\left(\boldsymbol{X}_{p_{T}},1\right)\leq \rho_{C_{2p_T}}$. Hence, Theorem \ref{thm:rho-mixing} can be proved if we can show that 
 $\underset{p_{T}\rightarrow\infty}{\lim}\phi_{C}^{2}\left(p_{T}\right)<\gamma$ yields $\rho_{C_{2p_T}}<\kappa<1$, where $\gamma$ does not depend on $p_{T}$.
 
The following proof is an extension of Proposition 2 in \cite{Remillard2012CopulaTimeSeries}.

Because$\underset{p_{T}\rightarrow\infty}{\lim}\phi^{2}\left(p_{T}\right)<\gamma$,
we will have $\underset{p_{T}\rightarrow\infty}{\lim}\phi_{C}^{2}\left(p_{T}\right)<\gamma$,
where $\gamma$ doesn't depend on $p_{T}$. It implies that $\phi_{C}^{2}\left(p_{T}\right)$
is bounded by some absolute constant for any $T\in\mathbb{N}^{+}$.
According to the Proposition 2 in \cite{Remillard2012CopulaTimeSeries}, there exists some absolute constant $\kappa>0$ not dependent on $T$, such that
$\rho_{C_{2p_T}}<\kappa<1$. Then our theorem is  proved.
\end{proof}

\subsection{High Dimensional Time Series Model}
In the following high dimensional time series, we consider the triangular array setting. Stationary Markov Chain $\left\{ \boldsymbol{X}_{p_{T}} (t)\right\} _{t\in\mathbb{Z}}$
is generated by VAR(1) and VARMA(1,1) respectively. However, For each $T\in\mathbb{N}^+$, we only observe a length $T$ fragment, $\left\{\boldsymbol{X}_{p_T}(t),t\in [T]\right\}$
of the time series $\left\{ \boldsymbol{X}_{p_{T}} (t)\right\} _{t\in\mathbb{Z}}$. And as $T$ goes to infinity, the dimension of $ \boldsymbol{X}_{p_{T}} (t)$ goes infinity too.


By showing that  $\underset{p_T\rightarrow\infty}{\lim}\phi^{2}(p_T)<M<\infty$, where $M$ does not depend on $T$, we will prove for any $T$, the Markov Chain $\left\{ \boldsymbol{X}_{p_{T}} (t)\right\} _{t\in\mathbb{Z}}$ generated by VAR(1) or VARMA(1,1) satisfies $\rho$-mixing.

\subsubsection{VAR(1)}
\begin{theorem}
Suppose $\left\{ \boldsymbol{X}_{p_{T}} (t)\right\} _{t\in\mathbb{Z}}$
be a $p_{T}-$dimensional multivariate stationary Markov Chain generated
by VAR(1) model:
\begin{equation}\label{model:VAR}
\boldsymbol{X}_{p_{T}}(t)=A_{p_T}\boldsymbol{X}_{p_{T}}(t-1)+\xi_{p_T}(t),
\end{equation}
where, $\xi_{p_T}(t)\sim N_{p_{T}}(0,\Sigma_{\xi_{p_T}}),\cov\{\xi_{p_T}(t),\xi_{p_T}(s)\}=0\text{ if }t\neq s$.
There is 
\[
\underset{p_T\rightarrow\infty}{\lim}\phi^{2}(p_T)<M<\infty,
\]
if we assume
\begin{itemize}
    \item {\bf Assumption (A1)} $p_{T}\times p_{T}$ matrix $A_{p_T}$ with fixed rank $k$;
    \item {\bf Assumption (A2)}  $\lambda_{\min}(\Sigma_{\xi_{p_T}})>\delta> 0$, $\delta $ is an absolute constant not depend on $p_T$;
    \item and {\bf Assumption (A3)} $\lambda_{\max}(\Gamma_{X_{p_T}})<\zeta<\infty$, $\zeta$ is an absolute constant not depend on $p_T$.
\end{itemize}
\end{theorem}
\begin{proof}
Recall Definition $\ref{def:phiSq}$,
if we can show that based on above assumptions $\phi^{2}\left(p_{T}\right)$
is bounded by some absolute constants not dependent on $T$, then the
theorem is proved. 
Denote $\Gamma_{X_{p_{T}}}=\var\left\{\boldsymbol{X}_{p_{T}}(t)\right\}$,
$\Gamma_{X_{p_{T}}}(1)=\cov\left\{\boldsymbol{X}_{p_{T}}(t+1),\boldsymbol{X}_{p_{T}}(t)\right\}=A_{p_{T}}\Gamma_{X_{p_{T}}}$,
$\boldsymbol{Z}_{2p_{T}}=\left(\boldsymbol{X}_{p_{T}}(t),\boldsymbol{X}_{p_{T}}(t+1)\right)$.
Because $\left\{ \boldsymbol{X}_{p_{T}}(t)\right\} _{t\in\mathbb{Z}}$
is a stationary time series, there exists
\begin{equation}
\Gamma_{X_{p_{T}}}=A_{p_{T}}\Gamma_{X_{p_{T}}}A_{p_{T}}^{\intercal}+\Sigma_{\xi_{p_{T}}},\label{eq:stationaryCondition_VAR}
\end{equation}
and the marginal distributions for $\boldsymbol{X}_{p_{T}}(t)$ and
$\left(\boldsymbol{X}_{p_{T}}(t),\boldsymbol{X}_{p_{T}}(t+1)\right)$
will remain the same respectively for any $t\in\mathbb{Z}$.
Then we will have 
\[
\boldsymbol{X}_{p_{T}}(t)\sim N_{p_{T}}\left(0,\Gamma_{X_{p_{T}}}\right),
\]
\[
\boldsymbol{Z}_{2p_{T}}\sim N_{2p_{T}}\left(0,\Sigma_{2p_{n}}\right),
\]
where
\begin{align}
\Sigma_{2p_{T}} & =\begin{pmatrix}\Gamma_{X_{p_{T}}} & \Gamma_{X_{p_{T}}}A_{p_{T}}^{\intercal}\\
A_{p_{T}}\Gamma_{X_{p_{T}}} & \Gamma_{X_{p_{T}}}
\end{pmatrix}.\label{eq:Sigma_VAR}
\end{align}
The density of $\boldsymbol{Z}_{2p_{T}}$ is 
\begin{align*}
h\left(\boldsymbol{z}_{2p_{T}}\right) & =\frac{1}{\left(2\pi\right)^{\frac{2p_{T}}{2}}\left|\Sigma_{2p_{T}}\right|^{\frac{1}{2}}}\exp\left\{ -\frac{1}{2}\boldsymbol{z}_{2p_{T}}^{\intercal}\Sigma_{2p_{T}}^{-1}\boldsymbol{z}_{2p_{T}}\right\} .
\end{align*}
The product of the density of $\boldsymbol{X}_{p_{T}}(t)$ and $\boldsymbol{X}_{p_{T}}(t+1)$    is 
\begin{align*}
p\left\{\boldsymbol{x}_{p_{T}}\left(t\right)\right\}p\left\{\boldsymbol{x}_{p_{T}}\left(t+1\right)\right\} & =\left[\frac{1}{\left(2\pi\right)^{\frac{p_{T}}{2}}\left|\Gamma_{X_{p_{T}}}\right|^{\frac{1}{2}}}\exp\left\{ -\frac{1}{2}\boldsymbol{x}_{p_{T}}\left(t\right)^{\intercal}\Gamma_{X_{p_{T}}}^{-1}\boldsymbol{x}_{p_{T}}\left(t\right)\right\} \right]^{2}\\
 & =\frac{1}{\left(2\pi\right)^{\frac{2p_{T}}{2}}\left|\widetilde{\Sigma}_{2p_{T}}\right|^{\frac{1}{2}}}\exp\left\{ -\frac{1}{2}\boldsymbol{z}_{2p_{T}}^{\intercal}\widetilde{\Sigma}_{2p_{T}}^{-1}\boldsymbol{z}_{2p_{T}}\right\} ,
\end{align*}
where 
\begin{equation}
\widetilde{\Sigma}_{2p_{T}}=\left(\begin{array}{cc}
\Gamma_{X_{p_{T}}} & \boldsymbol{0}\\
\boldsymbol{0} & \Gamma_{X_{p_{T}}}
\end{array}\right).\label{eq:title_Sigma_VAR}
\end{equation}
Therefore, 
\begin{align}
\frac{h\left(\boldsymbol{x}_{p_{T}}\left(t\right),\boldsymbol{x}_{p_{T}}\left(t+1\right)\right)}{p\left(\boldsymbol{x}_{p_{T}}\left(t\right)\right)p\left(\boldsymbol{x}_{p_{T}}\left(t+1\right)\right)} & =\frac{\left(\frac{1}{\left(2\pi\right)^{\frac{2p_{T}}{2}}\left|\Sigma_{2p_{T}}\right|^{\frac{1}{2}}}\exp\left\{ -\frac{1}{2}\boldsymbol{z}_{2p_{T}}^{\intercal}\Sigma_{2p_{T}}^{-1}\boldsymbol{z}_{2p_{T}}\right\} \right)^{2}}{\frac{1}{\left(2\pi\right)^{\frac{2p_{T}}{2}}\left|\widetilde{\Sigma}_{2p_{T}}\right|^{\frac{1}{2}}}\exp\left\{ -\frac{1}{2}\boldsymbol{z}_{2p_{T}}^{\intercal}\widetilde{\Sigma}_{2p_{T}}^{-1}\boldsymbol{z}_{2p_{T}}\right\} }\nonumber \\
 & =\frac{\left|\widetilde{\Sigma}_{2p_{T}}\right|^{\frac{1}{2}}}{\left(2\pi\right)^{\frac{2p_{T}}{2}}\left|\Sigma_{2p_{T}}\right|}\exp\left\{ -\frac{1}{2}\boldsymbol{z}_{2p_{T}}^{\intercal}\left(2\Sigma_{2p_{T}}^{-1}-\widetilde{\Sigma}_{2p_{T}}^{-1}\right)\boldsymbol{z}_{2p_{T}}^{\intercal}\right\} .\label{eq:kernel}
\end{align}
From Lemma \ref{A:InvCovDiff_def}, we know $2\Sigma_{2p_{T}}^{-1}-\widetilde{\Sigma}_{2p_{T}}^{-1}$
is a positive definite matrix with fixed dimension. Plugging (\ref{eq:kernel})
into $\phi^{2}(p_{T})$ yields 
\begin{align}
\phi^{2}\left(p_{T}\right) & =-1+\int\frac{h\left(\boldsymbol{x}_{p_{T}}\left(t\right),\boldsymbol{x}_{p_{T}}\left(t+1\right)\right)}{p\left(\boldsymbol{x}_{p_{T}}\left(t\right)\right)p\left(\boldsymbol{x}_{p_{T}}\left(t+1\right)\right)}d\boldsymbol{z}_{2p_{T}}\nonumber \\
 & =-1+\frac{\left|\widetilde{\Sigma}_{2p_{T}}\right|^{\frac{1}{2}}\left|\left(2\Sigma_{2p_{T}}^{-1}-\widetilde{\Sigma}_{2p_{T}}^{-1}\right)^{-1}\right|^{\frac{1}{2}}}{\left|\Sigma_{2p_{T}}\right|}.\label{eq:phiSq_VAR}
\end{align}
The following will show that (\ref{eq:phiSq_VAR}) is always bounded
by some absolute constant $M$ independent of $T$. 
Since $\left|\left(2\Sigma_{2p_{T}}^{-1}-\widetilde{\Sigma}_{2p_{T}}^{-1}\right)^{-1}\right|=\frac{1}{\left|\left(2\Sigma_{2p_{T}}^{-1}-\widetilde{\Sigma}_{2p_{T}}^{-1}\right)\right|}$,
$\left|\widetilde{\Sigma}_{2p_{T}}\right|^{\frac{1}{2}}=\frac{1}{\left|\widetilde{\Sigma}_{2p_{T}}^{-1}\right|^{\frac{1}{2}}}$
, then 
\begin{equation}
\frac{\left|\widetilde{\Sigma}_{2p_{T}}\right|^{\frac{1}{2}}\left|\left(2\Sigma_{2p_{T}}^{-1}-\widetilde{\Sigma}_{2p_{T}}^{-1}\right)^{-1}\right|^{\frac{1}{2}}}{\left|\Sigma_{2p_{T}}\right|}=\frac{1}{\left|\left(2\Sigma_{2p_{T}}^{-1}-\widetilde{\Sigma}_{2p_{T}}^{-1}\right)\right|^{\frac{1}{2}}\left|\Sigma_{2p_{T}}\right|\left|\widetilde{\Sigma}_{2p_{T}}^{-1}\right|^{\frac{1}{2}}}.\label{eq:phiSq_VAR_simply1}
\end{equation}
Rearrange the denominator of (\ref{eq:phiSq_VAR_simply1}) by the property
of matrix determinant, we further have 
\begin{align}
\left|\left(2\Sigma_{2p_{T}}^{-1}-\widetilde{\Sigma}_{2p_{T}}^{-1}\right)\right|^{\frac{1}{2}}\left|\Sigma_{2p_{T}}\right|\left|\widetilde{\Sigma}_{2p_{T}}^{-1}\right|^{\frac{1}{2}} & =\left|\left(2I_{2p_{T}}-\widetilde{\Sigma}_{2p_{T}}^{-1}\Sigma_{2p_{T}}\right)\Sigma_{2p_{T}}\widetilde{\Sigma}_{2p_{T}}^{-1}\right|^{\frac{1}{2}}.\label{eq:phiSq_VAR_simply2}
\end{align}
Recalling (\ref{eq:Sigma_VAR}),(\ref{eq:title_Sigma_VAR}), we can compute
the matrix determinant involved in  (\ref{eq:phiSq_VAR_simply2}),
\begin{equation}
\Sigma_{2p_{T}}\widetilde{\Sigma}_{2p_{T}}^{-1}=\begin{pmatrix}\Gamma_{X_{p_{T}}} & \Gamma_{X_{p_{T}}}A_{p_{T}}^{\intercal}\\
A_{p_{T}}\Gamma_{X_{p_{T}}} & \Gamma_{X_{p_{T}}}
\end{pmatrix}\left(\begin{array}{cc}
\Gamma_{X_{p_{T}}} & \boldsymbol{0}\\
\boldsymbol{0} & \Gamma_{X_{p_{T}}}
\end{array}\right)^{-1}=\left(\begin{array}{cc}
I_{p_{T}} & \Gamma_{X_{p_{T}}}A_{p_{T}}^{\intercal}\Gamma_{X_{p_{T}}}^{-1}\\
A_{p_{T}} & I_{p_{T}}
\end{array}\right),
\end{equation}

\begin{equation}
\widetilde{\Sigma}_{2p_{T}}^{-1}\Sigma_{2p_{T}}=\left(\Sigma_{2p_{T}}\widetilde{\Sigma}_{2p_{T}}^{-1}\right)^{\intercal}=\left(\begin{array}{cc}
I_{p_{T}} & A_{p_{T}}^{\intercal}\\
\Gamma_{X_{p_{T}}}^{-1}A_{p_{T}}\Gamma_{X_{p_{T}}} & I_{p_{T}}
\end{array}\right)\label{eq:InvSigmaTitleSigma}
\end{equation}
Noticing $\left|2I_{2p_{T}}-\widetilde{\Sigma}_{2p_{T}}^{-1}\Sigma_{2p_{T}}\right|=\left|\left(\begin{array}{cc}
I_{p_{T}} & -A_{p_{T}}^{\intercal}\\
-\Gamma_{X_{p_{T}}}^{-1}A_{p_{T}}\Gamma_{X_{p_{T}}} & I_{p_{T}}
\end{array}\right)\right|=\left|\left(\begin{array}{cc}
I_{p_{T}} & A_{p_{T}}^{\intercal}\\
\Gamma_{X_{p_{T}}}^{-1}A_{p_{T}}\Gamma_{X_{p_{T}}} & I_{p_{T}}
\end{array}\right)\right|$, and it is exactly the determinant of $\widetilde{\Sigma}_{2p_{T}}^{-1}\Sigma_{2p_{T}}$.
Accordingly, applying the transpose of one matrix will not change
its determinant, we further know 
\begin{equation}\label{eq:two_parts_equal_VAR}
\left|2I_{2p_{T}}-\widetilde{\Sigma}_{2p_{T}}^{-1}\Sigma_{2p_{T}}\right|=\left|\widetilde{\Sigma}_{2p_{T}}^{-1}\Sigma_{2p_{T}}\right|=\left|\Sigma_{2p_{T}}\widetilde{\Sigma}_{2p_{T}}^{-1}\right|.
\end{equation}
By (\ref{eq:InvSigmaTitleSigma}),(\ref{eq:two_parts_equal_VAR}),
we deduce (\ref{eq:phiSq_VAR_simply2})
\begin{align}
\left|\left(2I_{2p_{T}}-\widetilde{\Sigma}_{2p_{T}}^{-1}\Sigma_{2p_{T}}\right)\Sigma_{2p_{T}}\widetilde{\Sigma}_{2p_{T}}^{-1}\right|^{\frac{1}{2}} & =\left|\widetilde{\Sigma}_{2p_{T}}^{-1}\Sigma_{2p_{T}}\right|\nonumber \\
 & =\left|I_{p_{T}}-A_{p_{T}}^{\intercal}\Gamma_{X_{p_{T}}}^{-1}A_{p_{T}}\Gamma_{X_{p_{T}}}\right|.\label{eq:denominator1}
\end{align}
By \ref{A:AB_pd}, we know all the eigenvalues of
$A_{p_{T}}^{\intercal}\Gamma_{X_{p_{T}}}^{-1}A_{p_{T}}\Gamma_{X_{p_{T}}}$
are real, and as the same as the eigenvalues of $\Gamma_{X_{p_{T}}}^{-1}A_{p_{T}}\Gamma_{X_{p_{T}}}A_{p_{T}}^{\intercal}$.
Rewrite (\ref{eq:phiSq_VAR}) according to (\ref{eq:phiSq_VAR_simply1})(\ref{eq:phiSq_VAR_simply2})(\ref{eq:denominator1})
\begin{equation}
\phi^{2}\left(p_{T}\right)=-1+\frac{1}{\left|I_{p_{T}}-\Gamma_{X_{p_{T}}}^{-1}A_{p_{T}}\Gamma_{X_{p_{T}}}A_{p_{T}}^{\intercal}\right|}.\label{eq:phiSq_simple_VAR}
\end{equation}
If we want to prove $\phi^2(p_T)$ is bounded by $M$, it is equivalent to show that the minimum eigenvalue of  $I_{p_{T}}-\Gamma_{X_{p_{T}}}^{-1}A_{p_{T}}\Gamma_{X_{p_{T}}}A_{p_{T}}^{\intercal}$ is larger than a non-zero absolute constant.
Noticing the stationary VAR(1) satisfying condition (\ref{eq:stationaryCondition_VAR}),
it further implies 
\[
\Sigma_{\xi_{p_{T}}}=\Gamma_{X_{p_{T}}}-A_{p_{T}}\Gamma_{X_{p_{T}}}A_{p_{T}}^{\intercal}=\Gamma_{X_{p_{T}}}\left(I_{p_{T}}-\Gamma_{X_{p_{T}}}^{-1}A_{p_{T}}\Gamma_{X_{p_{T}}}A_{p_{T}}^{\intercal}\right).
\]
then 
\begin{equation}
I_{p_{T}}-\Gamma_{X_{p_{T}}}^{-1}A_{p_{T}}\Gamma_{X_{p_{T}}}A_{p_{T}}^{\intercal}=\Gamma_{X_{p_{T}}}^{-1}\Sigma_{\xi_{p_{T}}}.\label{eq:GammaGammaxi_VAR}
\end{equation}
 It also shows that $I_{p_{T}}-\Gamma_{X_{p_{T}}}^{-1}A_{p_{T}}\Gamma_{X_{p_{T}}}A_{p_{T}}^{\intercal}$
is a positive definite matrix, and it is easy to tell the largest eigenvalue of  $I_{p_{T}}-\Gamma_{X_{p_{T}}}^{-1}A_{p_{T}}\Gamma_{X_{p_{T}}}A_{p_{T}}^{\intercal}$ is $1$
because of the low rank assumption (\textbf{A1}) of $A_{p_{T}}$. Then the question
turns to prove $\lambda_{\min}\left(\Gamma_{X_{p_{T}}}^{-1}\Sigma_{\xi_{p_{T}}}\right)>\epsilon>0$.
\[
\lambda_{\min}\left(\Gamma_{X_{p_{T}}}^{-1}\Sigma_{\xi_{p_{T}}}\right)=\frac{1}{\Vert\left(\Gamma_{X_{p_{T}}}^{-1}\Sigma_{\xi_{p_{T}}}\right)^{-1}\Vert}=\frac{1}{\norm{\Sigma_{\xi_{p_{T}}}^{-1}\Gamma_{X_{p_{T}}}}}.
\]
Applying the operator norm inequality, $\Vert\Sigma_{\xi_{p_{T}}}^{-1}\Gamma_{X_{p_{T}}}\Vert\leq\Vert\Sigma_{\xi_{p_{T}}}^{-1}\Vert\Vert\Gamma_{X_{p_{T}}}\Vert=\frac{\lambda_{\max}\left(\Gamma_{X_{p_{T}}}\right)}{\lambda_{\min}\left(\Sigma_{\xi_{p_{T}}}\right)}$.
According to our Assumption (\textbf{A2,A3}), we have 
\begin{equation}
\lambda_{\min}\left(\Gamma_{X_{p_{T}}}^{-1}\Sigma_{\xi_{p_{T}}}\right)>\frac{\zeta}{\delta}>0,\label{eq:lowBound_VAR}
\end{equation}
which also means $\lambda_{\min}\left(I_{p_{T}}-\Gamma_{X_{p_{T}}}^{-1}A_{p_{T}}\Gamma_{X_{p_{T}}}A_{p_{T}}^{\intercal}\right)>\frac{\zeta}{\delta}>0$. 
Noticing the determinant 
\[
\left|I_{p_{T}}-\Gamma_{X_{p_{T}}}^{-1}A_{p_{T}}\Gamma_{X_{p_{T}}}A_{p_{T}}^{\intercal}\right|=\prod_{i=1}^{p_{T}}\lambda_{i}\left(I_{p_{T}}-\Gamma_{X_{p_{T}}}^{-1}A_{p_{T}}\Gamma_{X_{p_{T}}}A_{p_{T}}^{\intercal}\right).
\]
Applying $\text{rank}(A_{p_{T}})=k$ assumption (\textbf{A1}) and $\lambda_{\min}\left(I_{p_{T}}-\Gamma_{X_{p_{T}}}^{-1}A_{p_{T}}\Gamma_{X_{p_{T}}}A_{p_{T}}^{\intercal}\right)>\frac{\zeta}{\delta}$,
we have 
\begin{align*}
\prod_{i=1}^{p_T}\lambda_{i}\left(I_{p_{T}}-\Gamma_{X_{p_{T}}}^{-1}A_{p_{T}}\Gamma_{X_{p_{T}}}A_{p_{T}}^{\intercal}\right) & =\prod_{i=p_T-k+1}^{p_T}\lambda_{i}\left(I_{p_{T}}-\Gamma_{X_{p_{T}}}^{-1}A_{p_{T}}\Gamma_{X_{p_{T}}}A_{p_{T}}^{\intercal}\right)\\
 & >\left(\frac{\zeta}{\delta}\right)^{k}
\end{align*}
Therefore, $\phi^{2}\left(p_{T}\right)<-1+\left(\frac{\delta}{\zeta}\right)^{k}<M$
for any $T\in\mathbb{N}^{+}$, which completes our proof. 
\end{proof}

\subsubsection{VARMA(1,1)}
Borrow the low-rank modeling strategy from \cite{Basu2014}.
Suppose a $p_{T}$-dimensional multivariate time series $\left\{ \boldsymbol{X}_{p_{T}} (t)\right\} _{t\in\mathbb{Z}}$ is driven by $\left\{\boldsymbol{F}(t):=\left(F_{1}(t),\ldots,F_{k}(t)\right)^{\intercal} \right\}_{t\in\mathbb{Z}}$, which is $k-$dimensional $\left(k\ll p_{T}\right)$ time series following a VAR(1) process, then the model (\textbf{M1}) can be interpreted as 
\begin{align}
\boldsymbol{X}_{p_{T}}\left(t\right) & =\Lambda_{p_T}\boldsymbol{F}\left(t\right)+\xi_{p_T}(t),\mathbb\quad\xi_{p_T}(t)\sim N_{p_{T}}\left(0,\Sigma_{\xi_{p_T}}\right),\mathbb\quad \cov\left(\xi_{p_T}(t),\xi_{p_T}(s)\right)=0\text{ if }t\neq s \label{M1:1} \\
\boldsymbol{F}\left(t\right) & =H\boldsymbol{F}\left(t-1\right)+\eta(t),\mathbb\quad\eta(t)\sim N_{k}\left(0,\Sigma_{\eta}\right),\mathbb\quad \cov\left(\eta(t),\eta(s)\right)=0\text{ if }t\neq s \label{M1:2}.
\end{align}
Assume the $p_{T}\times k$ matrix $\Lambda_{p_T}$ having full column rank
$k$, hence there exists its left inverse $\Lambda_{p_T}^{-1}$ such that $\Lambda_{p_T}^{-1}\Lambda_{p_T}=\Ib_{k}$. 
Hence, we will have the VARMA(1,1) model,
\begin{equation}\label{eq:build low rank}
\boldsymbol{X}_{p_{T}}\left(t\right) =L_{p_T}\boldsymbol{X}_{p_{T}}\left(t-1\right)+\epsilon_{p_T}(t),
\end{equation}
 where $L_{p_T}=\Lambda_{p_{T}} H \Lambda_{p_{T}}^{-1}$ with rank $k$, and the new error $\epsilon_{p_T}(t)=\Lambda_{p_T}(t)\eta(t) +\xi_{p_T}(t)-L_{p_T}(t)\xi_{p_T}(t-1)$  with  $\epsilon_{p_T}(t)\sim N_{p_{T}}\left(0,\Sigma_{\epsilon_{p_T}}\right)$.
\\ 
Having the variance-covariance matrix of $\boldsymbol{X}_{p_T}(t)$ from (\ref{M1:1})
\begin{equation}\label{eq:VarX_1}
\Gamma_{X_{p_T}} = \var(\boldsymbol{X}_{p_T}(t)) = \Lambda_{p_T}\Gamma_{F}\Lambda_{p_T}^\intercal+\Sigma_{\xi_{p_T}}.
\end{equation}
Also, we can compute the variance-covariance matrix of $\boldsymbol{X}_{p_T}(t)$  from (\ref{eq:build low rank}),
\begin{equation}\label{eq:VarX_2}
\begin{aligned}
\Gamma_{X_{p_{T}}}  &=L_{p_{T}}\Gamma_{X_{p_{T}}}L_{p_{T}}^\intercal +\Lambda_{p_{T}}\Sigma_{\eta}\Lambda_{p_{T}}^\intercal+\Sigma_{\xi_{p_{T}}}-L_{p_{T}}\Sigma_{\xi_{p_{T}}}L_{p_{T}}^\intercal\\
& = L_{p_{T}}\Gamma_{X_{p_{T}}}L_{p_{T}}^\intercal +\Lambda_{p_{T}}\Sigma_{\eta}\Lambda_{p_{T}}^\intercal+ (I_{p_T} -L_{p_{T}})\Sigma_{\xi_{p_{T}}} (I_{p_T} -L_{p_{T}})^\intercal.
\end{aligned}
\end{equation}
By (\ref{M1:1}), we can also have the covariance matrix between $\boldsymbol{X}_{p_T}(t+1)$, and  $\boldsymbol{X}_{p_T}(t)$,
\[
\Gamma_{X_{p_{T}}}(1) =\Lambda_{p_{T}}\Gamma_{F}\left(1\right)\Lambda_{p_{T}}^\intercal=\Lambda_{p_{T}} H\Gamma_{F}\Lambda_{p_{T}}^\intercal.
\]
$I_{k}$ can be decomposed as $\Lambda_{p_T}^{-1}\Lambda_{p_T}$, hence the above equation can be represented as $\Lambda_{p_{T}} H\Gamma_{F}\Lambda_{p_{T}}^\intercal = \Lambda_{p_{T}} H\Lambda_{p_{T}}^{-1}\Lambda_{p_{T}}\Gamma_{F}\Lambda_{p_{T}}^\intercal$. Noticing  (\ref{eq:VarX_1}) and $L_{p_T}=\Lambda_{p_{T}} H \Lambda_{p_{T}}^{-1}$, it yields 
\begin{equation}\label{eq:varX(1)}
\Gamma_{X_{p_{T}}}(1) = L_{p_{T}}\left(\Gamma_{X_{p_{T}}}-\Sigma_{\xi_{p_{T}}}\right).
\end{equation}

\begin{theorem}
Suppose $\left\{ \boldsymbol{X}_{p_{T}} (t)\right\} _{t\in\mathbb{Z}}$
be a $p_{T}-$dimensional multivariate stationary Markov Chain generated by (\ref{M1:1}) and (\ref{M1:2}). 
There is
\[
\underset{p_T\rightarrow\infty}{\lim}\phi^{2}(p_T)<M<\infty,
\]
if we assume, 
\begin{itemize}
    \item {\bf Assumption (A1)} $L_{p_T}=\Lambda_{p_{T}} H \Lambda_{p_{T}}^{-1}$ with fixed rank $k$, and $\norm{L_{p_T}}<\mu<1$, $\mu$ is an absolute constant not depend on $p_T$;
    \item {\bf Assumption (A2)}  $\lambda_{\min}(\Sigma_{\xi_{p_T}})>\delta> 0$, $\delta $ is an absolute constant not depend on $p_T$;
    \item and {\bf Assumption (A3)} $\lambda_{\max}(\Gamma_{X_{p_T}})<\zeta<\infty$, $\zeta$ is an absolute constant not depend on $p_T$.
\end{itemize}
\end{theorem}

\begin{proof}
Similarly as VAR(1), we have 
\begin{align}
\phi^{2}\left(p_{T}\right) & =-1+\frac{\left|\widetilde{\Sigma}_{2p_{T}}\right|^{\frac{1}{2}}\left|\left(2\Sigma_{2p_{T}}^{-1}-\widetilde{\Sigma}_{2p_{T}}^{-1}\right)^{-1}\right|^{\frac{1}{2}}}{\left|\Sigma_{2p_{T}}\right|}\nonumber \\
 & =-1+\frac{1}{\left|\left(2I_{2p_{T}}-\widetilde{\Sigma}_{2p_{T}}^{-1}\Sigma_{2p_{T}}\right)\Sigma_{2p_{T}}\widetilde{\Sigma}_{2p_{T}}^{-1}\right|^{\frac{1}{2}}},\label{eq:phi_VARMA}
\end{align}
where,
\begin{equation}
\Sigma_{2p_{T}}=\begin{pmatrix}\Gamma_{X_{p_{T}}} & \left(\Gamma_{X_{p_{T}}}-\Sigma_{\xi_{p_{T}}}\right)L_{p_{T}}^{\intercal}\\
L_{p_{T}}\left(\Gamma_{X_{p_{T}}}-\Sigma_{\xi_{p_{T}}}\right) & \Gamma_{X_{p_{T}}}
\end{pmatrix}\label{eq:Sigma_VARAM},
\end{equation}
\begin{equation}
\widetilde{\Sigma}_{2p_{T}}=\left(\begin{array}{cc}
\Gamma_{X_{p_{T}}} & \boldsymbol{0}\\
\boldsymbol{0} & \Gamma_{X_{p_{T}}}
\end{array}\right).\label{eq:Sigma_title_VARAM}
\end{equation}
 Therefore, 
\begin{align}
\Sigma_{2p_{T}}\widetilde{\Sigma}_{2p_{T}}^{-1} & =\begin{pmatrix}\Gamma_{X_{p_{T}}} & \left(\Gamma_{X_{p_{T}}}-\Sigma_{\xi_{p_{T}}}\right)L_{p_{T}}^{\intercal}\\
L_{p_{T}}\left(\Gamma_{X_{p_{T}}}-\Sigma_{\xi_{p_{T}}}\right) & \Gamma_{X_{p_{T}}}
\end{pmatrix}\left(\begin{array}{cc}
\Gamma_{X_{p_{T}}} & \boldsymbol{0}\\
\boldsymbol{0} & \Gamma_{X_{p_{T}}}
\end{array}\right)^{-1}\nonumber \\
 & =\left(\begin{array}{cc}
I_{p_{T}} & \left(\Gamma_{X_{p_{T}}}-\Sigma_{\xi_{p_{T}}}\right)L_{p_{T}}^{\intercal}\Gamma_{X_{p_{T}}}^{-1}\\
L_{p_{T}}\left(\Gamma_{X_{p_{T}}}-\Sigma_{\xi_{p_{T}}}\right)\Gamma_{X_{p_{T}}}^{-1} & I_{p_{T}}
\end{array}\right),\label{eq:Sigam_SigmaTitle_VARMA}
\end{align}
and 
\begin{equation}
\widetilde{\Sigma}_{2p_{T}}^{-1}\Sigma_{2p_{T}}=\left(\Sigma_{2p_{T}}\widetilde{\Sigma}_{2p_{T}}^{-1}\right)^{\intercal}=\left(\begin{array}{cc}
I_{p_{T}} & \Gamma_{X_{p_{T}}}^{-1}\left(\Gamma_{X_{p_{T}}}-\Sigma_{\xi_{p_{T}}}\right)L_{p_{T}}^{\intercal}\\
\Gamma_{X_{p_{T}}}^{-1}L_{p_{T}}\left(\Gamma_{X_{p_{T}}}-\Sigma_{\xi_{p_{T}}}\right) & I_{p_{T}}
\end{array}\right).\label{eq:SigmaTitle_Sigam_VARMA}
\end{equation}
Same trick as VAR(1) (\ref{eq:two_parts_equal_VAR}), $\left|2I_{2p_{T}}-\widetilde{\Sigma}_{2p_{T}}^{-1}\Sigma_{2p_{T}}\right|=\left|\widetilde{\Sigma}_{2p_{T}}^{-1}\Sigma_{2p_{T}}\right|=\left|\Sigma_{2p_{T}}\widetilde{\Sigma}_{2p_{T}}^{-1}\right|$,
then we can further deduce the denominator of (\ref{eq:phi_VARMA})
as
\[
\left|\left(2I_{2p_{T}}-\widetilde{\Sigma}_{2p_{T}}^{-1}\Sigma_{2p_{T}}\right)\Sigma_{2p_{T}}\widetilde{\Sigma}_{2p_{T}}^{-1}\right|^{\frac{1}{2}}=\left|I_{p_{T}}-\Gamma_{X_{p_{T}}}^{-1}\left(\Gamma_{X_{p_{T}}}-\Sigma_{\xi_{p_{T}}}\right)L_{p_{T}}^{\intercal}\Gamma_{X_{p_{T}}}^{-1}L_{p_{T}}\left(\Gamma_{X_{p_{T}}}-\Sigma_{\xi_{p_{T}}}\right)\right|.
\]
And then rewrite (\ref{eq:phi_VARMA}) 
\begin{equation}
\phi^{2}\left(p_{T}\right)=-1+\frac{1}{\left|I_{p_{T}}-\Gamma_{X_{p_{T}}}^{-1}\left(\Gamma_{X_{p_{T}}}-\Sigma_{\xi_{p_{T}}}\right)L_{p_{T}}^{\intercal}\Gamma_{X_{p_{T}}}^{-1}L_{p_{T}}\left(\Gamma_{X_{p_{T}}}-\Sigma_{\xi_{p_{T}}}\right)\right|}.\label{eq:phiSq_simple_VARMA}
\end{equation}
$\Gamma_{X_{p_{T}}}^{-1}$ is positive definite, and $\left(\Gamma_{X_{p_{T}}}-\Sigma_{\xi_{p_{T}}}\right)L_{p_{T}}^{\intercal}\Gamma_{X_{p_{T}}}^{-1}L_{p_{T}}\left(\Gamma_{X_{p_{T}}}-\Sigma_{\xi_{p_{T}}}\right)$
is positive semidefinite, therefore, all the eigenvalues of their
product $\Gamma_{X_{p_{T}}}^{-1}\cdot\left(\Gamma_{X_{p_{T}}}-\Sigma_{\xi_{p_{T}}}\right)L_{p_{T}}^{\intercal}\Gamma_{X_{p_{T}}}^{-1}L_{p_{T}}\left(\Gamma_{X_{p_{T}}}-\Sigma_{\xi_{p_{T}}}\right)$
will be real non-negative numbers according to Lemma \ref{A:AB_pd}.\textbf{
}Same as VAR(1), we need to show $\lambda_{\min}\left(I_{p_{T}}-\Gamma_{X_{p_{T}}}^{-1}\left(\Gamma_{X_{p_{T}}}-\Sigma_{\xi_{p_{T}}}\right)L_{p_{T}}^{\intercal}\Gamma_{X_{p_{T}}}^{-1}L_{p_{T}}\left(\Gamma_{X_{p_{T}}}-\Sigma_{\xi_{p_{T}}}\right)\right)>M>0$,
which is also equivalent to show 
\begin{equation}
\Vert\Gamma_{X_{p_{T}}}^{-1}\left(\Gamma_{X_{p_{T}}}-\Sigma_{\xi_{p_{T}}}\right)L_{p_{T}}^{\intercal}\Gamma_{X_{p_{T}}}^{-1}L_{p_{T}}\left(\Gamma_{X_{p_{T}}}-\Sigma_{\xi_{p_{T}}}\right)\Vert<M'<1.\label{ineq:aim_prim}
\end{equation}
By $\Gamma_{X_{p_{T}}}^{-1}\left(\Gamma_{X_{p_{T}}}-\Sigma_{\xi_{p_{T}}}\right)L_{p_{T}}^{\intercal}\Gamma_{X_{p_{T}}}^{-1}L_{p_{T}}\left(\Gamma_{X_{p_{T}}}-\Sigma_{\xi_{p_{T}}}\right) = \Gamma_{X_{p_{T}}}^{-1}\left(\Gamma_{X_{p_{T}}}-\Sigma_{\xi_{p_{T}}}\right)\cdot L_{p_{T}}^{\intercal}\Gamma_{X_{p_{T}}}^{-1}L_{p_{T}}\Gamma_{X_{p_{T}}}\cdot\Gamma_{X_{p_{T}}}^{-1}\left(\Gamma_{X_{p_{T}}}-\Sigma_{\xi_{p_{T}}}\right)$, and applying operator norm inequality, we further have
\begin{equation}
\begin{aligned}
\Vert\Gamma_{X_{p_{T}}}^{-1}\left(\Gamma_{X_{p_{T}}}-\Sigma_{\xi_{p_{T}}}\right)L_{p_{T}}^{\intercal}\Gamma_{X_{p_{T}}}^{-1}L_{p_{T}}\left(\Gamma_{X_{p_{T}}}-\Sigma_{\xi_{p_{T}}}\right)\Vert & 
  \leq\Vert\Gamma_{X_{p_{T}}}^{-1}\left(\Gamma_{X_{p_{T}}}-\Sigma_{\xi_{p_{T}}}\right)\Vert^{2}\\
 & \cdot \Vert L_{p_{T}}^{\intercal}\Gamma_{X_{p_{T}}}^{-1}L_{p_{T}}\Gamma_{X_{p_{T}}}\Vert.\label{ineq:part_aim}
\end{aligned}
\end{equation}
Hence, it enough to show that $\Vert\Gamma_{X_{p_{T}}}^{-1}\left(\Gamma_{X_{p_{T}}}-\Sigma_{\xi_{p_{T}}}\right)\Vert<1$,
$\Vert L_{p_{T}}^{\intercal}\Gamma_{X_{p_{T}}}^{-1}L_{p_{T}}\Gamma_{X_{p_{T}}}\Vert<M'<1$
where $M'$ is a constant not dependent on $p_{T}$. 
$\Gamma_{X_{p_{T}}}-\Sigma_{\xi_{p_{T}}}=\Lambda_{p_{T}}\Gamma_{F}\Lambda_{p_{T}}^{\intercal}\succeq0$
By (\ref{eq:VarX_1}), hence all eigenvalues of $\Gamma_{X_{p_{T}}}^{-1}\left(\Gamma_{X_{p_{T}}}-\Sigma_{\xi_{p_{T}}}\right)=\left(I_{p_{T}}-\Gamma_{X_{p_{T}}}^{-1}\Sigma_{\xi_{p_{T}}}\right)$
should be real and non-negative \ref{A:AB_pd}. Therefore,
for $i\in\left[p_{T}\right]$ we have
\begin{align}
\lambda_{i}\left(I_{p_{T}}-\Gamma_{X_{p_{T}}}^{-1}\Sigma_{\xi_{p_{T}}}\right) & \geq0\label{ineq:I-gammaInvSigma}\\
\lambda_{i}\left(\Gamma_{X_{p_{T}}}^{-1}\Sigma_{\xi_{p_{T}}}\right) & >0.\label{ineq:GammaInvSigma}
\end{align}
(\ref{ineq:I-gammaInvSigma})(\ref{ineq:GammaInvSigma}) yield 
\[
0\leq\lambda_{i}\left(I_{p_{T}}-\Gamma_{X_{p_{T}}}^{-1}\Sigma_{\xi_{p_{T}}}\right)<1\mathbb\quad i\in\left[p_{T}\right],
\]
hence we have 
\begin{equation}
\Vert\Gamma_{X_{p_{T}}}^{-1}\left(\Gamma_{X_{p_{T}}}-\Sigma_{\xi_{p_{T}}}\right)\Vert<1.\label{eq:ineq.part1}
\end{equation}
Accordingly, $\Gamma_{X_{p_{T}}}^{-1}L_{p_{T}}\Gamma_{X_{p_{T}}}L_{p_{T}}^{\intercal}$
has the same real eigenvalues as $L_{p_{T}}^{\intercal}\Gamma_{X_{p_{T}}}^{-1}L_{p_{T}}\Gamma_{X_{p_{T}}}$
\ref{A:AB_pd}, then it is sufficient to show that
$\Vert\Gamma_{X_{p_{T}}}^{-1}L_{p_{T}}\Gamma_{X_{p_{T}}}L_{p_{T}}^{\intercal}\Vert<M'<1$. 
Recalling (\ref{eq:VarX_2}), 
we have
\begin{equation}
\Gamma_{X_{p_{T}}}^{-1}L_{p_{T}}\Gamma_{X_{p_{T}}}L_{p_{T}}^{\intercal}=I_{p_{T}}-\Gamma_{X_{p_{T}}}^{-1}\left\{ \Lambda_{p_{T}}\Sigma_{\eta}\Lambda_{p_{T}}^{\intercal}+\left(I_{p_{T}}-L_{p_{T}}\right)\Sigma_{\xi_{p_{T}}}\left(I_{p_{T}}-L_{p_{T}}\right)^{\intercal}\right\} .\label{eq:change_cov_VARMA}
\end{equation}
Then it is equivalent to show $\lambda_{\max}\left(I_{p_{T}}-\Gamma_{X_{p_{T}}}^{-1}\left\{ \Lambda_{p_{T}}\Sigma_{\eta}\Lambda_{p_{T}}^{\intercal}+\left(I_{p_{T}}-L_{p_{T}}\right)\Sigma_{\xi_{p_{T}}}\left(I_{p_{T}}-L_{p_{T}}\right)^{\intercal}\right\} \right)<M'<1$,
\begin{align*}
\Vert\Gamma_{X_{p_{T}}}^{-1}L_{p_{T}}\Gamma_{X_{p_{T}}}L_{p_{T}}^{\intercal}\Vert & =\lambda_{\max}\left(I_{p_{T}}-\Gamma_{X_{p_{T}}}^{-1}\left\{ \Lambda_{p_{T}}\Sigma_{\eta}\Lambda_{p_{T}}^{\intercal}+\left(I_{p_{T}}-L_{p_{T}}\right)\Sigma_{\xi_{p_{T}}}\left(I_{p_{T}}-L_{p_{T}}\right)^{\intercal}\right\} \right)\\
 & =1-\lambda_{\min}\left(\Gamma_{X_{p_{T}}}^{-1}\left\{ \Lambda_{p_{T}}\Sigma_{\eta}\Lambda_{p_{T}}^{\intercal}+\left(I_{p_{T}}-L_{p_{T}}\right)\Sigma_{\xi_{p_{T}}}\left(I_{p_{T}}-L_{p_{T}}\right)^{\intercal}\right\} \right).
\end{align*}
Using the property of eigenvalue, we have 
\begin{align*}
\lambda_{\min}\left(\Gamma_{X_{p_{T}}}^{-1}\left\{ \Lambda_{p_{T}}\Sigma_{\eta}\Lambda_{p_{T}}^{\intercal}+\left(I_{p_{T}}-L_{p_{T}}\right)\Sigma_{\xi_{p_{T}}}\left(I_{p_{T}}-L_{p_{T}}\right)^{\intercal}\right\} \right) & \geq\lambda_{\min}(\Gamma_{X_{p_{T}}}^{-1})\\
 & \mathbb\quad\cdot\lambda_{\min}\{\Lambda_{p_{T}}\Sigma_{\eta}\Lambda_{p_{T}}^{\intercal} \\
 & \mathbb\quad +\left(I_{p_{T}}-L_{p_{T}}\right)\Sigma_{\xi_{p_{T}}}\left(I_{p_{T}}-L_{p_{T}}\right)^{\intercal}\}.
\end{align*}
Using the Assumption (\textbf{A1) }that the rank of $L_{p_{T}}$ is
$k$, we have $\lambda_{\min}\left(\Lambda_{p_{T}}\Sigma_{\eta}\Lambda_{p_{T}}^{\intercal}\right)=0$.
Hence,\textbf{ }\\
\begin{align*}
\lambda_{\min}\left\{\Lambda_{p_{T}}\Sigma_{\eta}\Lambda_{p_{T}}^{\intercal}+\left(I_{p_{T}}-L_{p_{T}}\right)\Sigma_{\xi_{p_{T}}}\left(I_{p_{T}}-L_{p_{T}}\right)^{\intercal}\right\} & \geq\lambda_{\min}\left(\Lambda_{p_{T}}\Sigma_{\eta}\Lambda_{p_{T}}^{\intercal}\right)\\
&+\lambda_{\min}\left\{\left(I_{p_{T}}-L_{p_{T}}\right)\Sigma_{\xi_{p_{T}}}\left(I_{p_{T}}-L_{p_{T}}\right)^{\intercal}\right\}\\
 & =\lambda_{\min}\left\{\left(I_{p_{T}}-L_{p_{T}}\right)\Sigma_{\xi_{p_{T}}}\left(I_{p_{T}}-L_{p_{T}}\right)^{\intercal}\right\}.
\end{align*}
Noticing $\lambda_{\min}(\Gamma_{X_{p_{T}}}^{-1})=\frac{1}{\Vert\Gamma_{X_{p_{T}}}\Vert}$,
we further have
\[
\lambda_{\min}\left(\Gamma_{X_{p_{T}}}^{-1}\left\{ \Lambda_{p_{T}}\Sigma_{\eta}\Lambda_{p_{T}}^{\intercal}+\left(I_{p_{T}}-L_{p_{T}}\right)\Sigma_{\xi_{p_{T}}}\left(I_{p_{T}}-L_{p_{T}}\right)^{\intercal}\right\} \right)\geq\frac{\lambda_{\min}\left\{\left(I_{p_{T}}-L_{p_{T}}\right)\Sigma_{\xi_{p_{T}}}\left(I_{p_{T}}-L_{p_{T}}\right)^{\intercal}\right\}}{\Vert\Gamma_{X_{p_{T}}}\Vert}.
\]
Using 
\[
\lambda_{\min}\left(\left(I_{p_{T}}-L_{p_{T}}\right)\Sigma_{\xi_{p_{T}}}\left(I_{p_{T}}-L_{p_{T}}\right)^{\intercal}\right)=\frac{1}{\Big\Vert\left\{ \left(I_{p_{T}}-L_{p_{T}}\right)\Sigma_{\xi_{p_{T}}}\left(I_{p_{T}}-L_{p_{T}}\right)^{\intercal}\right\} ^{-1}\Big\Vert},
\]
and recalling the operator norm inequality, we have 
\begin{align*}
\frac{\lambda_{\min}\left\{\left(I_{p_{T}}-L_{p_{T}}\right)\Sigma_{\xi_{p_{T}}}\left(I_{p_{T}}-L_{p_{T}}\right)^{\intercal}\right\}}{\Vert\Gamma_{X_{p_{T}}}\Vert} & \geq\frac{1}{\Vert\Gamma_{X_{p_{T}}}\Vert\Vert\left(I_{p_{T}}-L_{p_{T}}\right)^{-1}\Vert^{2}\Vert\Sigma_{\xi_{p_{T}}}^{-1}\Vert}.
\end{align*}
Again, by $\frac{1}{\Vert\Sigma_{\xi_{p_{T}}}^{-1}\Vert}=\lambda_{\min}(\Sigma_{\xi_{p_{T}}})$,
and the Remark \ref{A:singularValue}, $\frac{1}{\Vert\left(I_{p_{T}}-L_{p_{T}}\right)^{-1}\Vert}=\sigma_{\min}(I_{p_{T}}-L_{p_{T}})\geq1-\Vert L_{p_{T}}\Vert$,
we have 
\[
\frac{1}{\Vert\Gamma_{X_{p_{T}}}\Vert\Vert\left(I_{p_{T}}-L_{p_{T}}\right)^{-1}\Vert^{2}\Vert\Sigma_{\xi_{p_{T}}}^{-1}\Vert}\geq\frac{\lambda_{\min}(\Sigma_{\xi_{p_{T}}})\left(1-\Vert L_{p_{T}}\Vert\right)^{2}}{\Vert\Gamma_{X_{p_{T}}}\Vert}.
\]
Applying the Assumption \textbf{A2 }and \textbf{A3, }we further have
\[
\frac{\lambda_{\min}(\Sigma_{\xi_{p_{T}}})\left(1-\Vert L_{p_{T}}\Vert\right)^{2}}{\Vert\Gamma_{X_{p_{T}}}\Vert}\geq\frac{\delta\left(1-\mu\right)^{2}}{\zeta},
\]
 hence $\lambda_{\min}\left(\Gamma_{X_{p_{T}}}^{-1}\left\{ \Lambda_{p_{T}}\Sigma_{\eta}\Lambda_{p_{T}}^{\intercal}+\left(I_{p_{T}}-L_{p_{T}}\right)\Sigma_{\xi_{p_{T}}}\left(I_{p_{T}}-L_{p_{T}}\right)^{\intercal}\right\} \right)\geq\frac{\delta\left(1-\mu\right)^{2}}{\zeta}$,
hence 
\begin{equation}
\Vert\Gamma_{X_{p_{T}}}^{-1}L_{p_{T}}\Gamma_{X_{p_{T}}}L_{p_{T}}^{\intercal}\Vert\leq1-\frac{\delta\left(1-\mu\right)^{2}}{\zeta}.\label{ineq:part2}
\end{equation}
By (\ref{ineq:part_aim})(\ref{eq:ineq.part1})(\ref{ineq:part2}),
we further have 
\begin{equation}
\Vert\Gamma_{X_{p_{T}}}^{-1}\left(\Gamma_{X_{p_{T}}}-\Sigma_{\xi_{p_{T}}}\right)L_{p_{T}}^{\intercal}\Gamma_{X_{p_{T}}}^{-1}L_{p_{T}}\left(\Gamma_{X_{p_{T}}}-\Sigma_{\xi_{p_{T}}}\right)\Vert\leq1-\frac{\delta\left(1-\mu\right)^{2}}{\zeta}.\label{ineq: aim_part_proved_VARMA}
\end{equation}
Back to the determinant of $\phi^{2}\left(p_{T}\right)$ (\ref{eq:phi_VARMA}),
with $\text{rank}\left(L_{p_{T}}\right)=k$ Assumption \textbf{A1}
and (\ref{ineq: aim_part_proved_VARMA}), we have 
\begin{align*}
\phi^{2}\left(p_{T}\right) & =-1+\frac{1}{\prod_{i=1}^{p_{T}}\left\{ 1-\lambda_{i}\left(\Gamma_{X_{p_{T}}}^{-1}\left(\Gamma_{X_{p_{T}}}-\Sigma_{\xi_{p_{T}}}\right)L_{p_{T}}^{\intercal}\Gamma_{X_{p_{T}}}^{-1}L_{p_{T}}\left(\Gamma_{X_{p_{T}}}-\Sigma_{\xi_{p_{T}}}\right)\right)\right\} }\\
 & \leq-1+\left(\frac{\zeta}{\delta\left(1-\mu\right)^{2}}\right)^{k}<\infty.
\end{align*}
This completes the proof. 
\end{proof}

\section*{Acknowledgement}
The work has been done during 2017, as an independent study project under supervision of Dr. Fang Han. The author is very grateful for the helpful guidance and discussion with Dr.~Han.

\bibliographystyle{apalike}
\bibliography{main}

\newpage


\appendix
\section{Auxiliary lemmas}

\begin{proposition}\label{A:prop_expodecay}
\cite{BradleyMixing} Suppose $a_{1},a_{2},\ldots,$ is a nonincreasing sequence
of number in $\left[0,\infty\right]$, suppose $M$ is a positive
integer. If $a_{nM}\longrightarrow0$ at least exponentially fast
as $n\rightarrow\infty$, then $a_{n}\rightarrow0$ at least exponentially
fast as $n\rightarrow\infty$.
\end{proposition}

\begin{proposition} \label{A:prop:markov rho}
Suppose $\left\{ \boldsymbol{X}_{p_{T}}(t)\right\} _{t\in\mathbb{Z}}$ is $p_T$-dimensional multivariate Markov Chain, then the following statements hold:
\begin{itemize}
    \item any $i,j\in\mathbb{Z}$, with $i\leq j$, $\rho (\mathcal{F}_{-\infty}^{i},\mathcal{F}_{j}^{\infty}) =\rho(\mathcal{F}_{j}^{j},\mathcal{F}_{i}^{i}) $;
    \item and $ \forall$ $i,j,k\in\mathbb{Z}$, with $i\leq j\leq k$, $\rho(\mathcal{F}_{i}^{i},\mathcal{F}_{k}^{k}) \leq\rho(\mathcal{F}_{i}^{i},\mathcal{F}_{j}^{j}) \cdot\rho(\mathcal{F}_{j}^{j},\mathcal{F}_{k}^{k})$.
\end{itemize}
\end{proposition}

\begin{proof}
Above results can be easily proved by the Theorem 7.2 and Theorem 7.4 from \cite{BradleyMixing}
\end{proof}

\begin{lemma}\label{A:InvCovDiff_def}
$n\times n$ matrix $A$ and $B$, suppose $A$ and  $\left(\begin{array}{cc}
A & B\\
B^\intercal & A
\end{array}\right)$ are positive definite,  then $$2\left(\begin{array}{cc}
A & B\\
B^\intercal & A
\end{array}\right)^{-1}-\left(\begin{array}{cc}
A & \boldsymbol{0}\\
\boldsymbol{0} & A
\end{array}\right)^{-1}$$ 
is always positive definite.
\end{lemma}

\begin{proof}
Since $A$ is p.d, then there is full rank matrix $A^{\frac{1}{2}}$ such that $A=A^{\frac{1}{2}}A^{\frac{1}{2}}$. We can decompose $\left(\begin{array}{cc}
A & B\\
B^\intercal & A
\end{array}\right) $ and $\left(\begin{array}{cc}
A & \boldsymbol{0}\\
\boldsymbol{0} & A
\end{array}\right) $ as following,
\begin{equation}
\left(\begin{array}{cc}
A & B\\
B^\intercal & A
\end{array}\right)  =\begin{pmatrix}A^{\frac{1}{2}} & \boldsymbol{0}\\
\boldsymbol{0} & A^{\frac{1}{2}}
\end{pmatrix}
\begin{pmatrix}I & (A^{\frac{1}{2}}) ^{-1}B(A^{\frac{1}{2}}) ^{-1}\\
(A^{\frac{1}{2}}) ^{-1}B^\intercal   (A^{\frac{1}{2}}) ^{-1} & I
\end{pmatrix}
\begin{pmatrix}A^{\frac{1}{2}} & \boldsymbol{0}\\
\boldsymbol{0} & A^{\frac{1}{2}}
\end{pmatrix}, 
\end{equation}

\begin{equation}
\left(\begin{array}{cc}
A & \boldsymbol{0}\\
\boldsymbol{0} & A
\end{array}\right) =\begin{pmatrix}A^{\frac{1}{2}} & \boldsymbol{0}\\
\boldsymbol{0} & A^{\frac{1}{2}}
\end{pmatrix}\begin{pmatrix}A^{\frac{1}{2}} & \boldsymbol{0}\\
\boldsymbol{0} & A^{\frac{1}{2}}
\end{pmatrix}.
\end{equation}

Take the inverse of above two matrix
\[
\left(\begin{array}{cc}
A & B\\
B^\intercal & A
\end{array}\right) ^{-1}=\begin{pmatrix}A^{\frac{1}{2}} & \boldsymbol{0}\\
\boldsymbol{0} & A^{\frac{1}{2}}
\end{pmatrix}^{-1}\begin{pmatrix}I & (A^{\frac{1}{2}}) ^{-1}B(A^{\frac{1}{2}}) ^{-1}\\
(A^{\frac{1}{2}}) ^{-1}B^\intercal   (A^{\frac{1}{2}}) ^{-1} & I
\end{pmatrix}^{-1}\begin{pmatrix}A^{\frac{1}{2}} & \boldsymbol{0}\\
\boldsymbol{0} & A^{\frac{1}{2}}
\end{pmatrix}^{-1},\label{eq:inv1}
\]
 
\[
\left(\begin{array}{cc}
A & \boldsymbol{0}\\
\boldsymbol{0} & A
\end{array}\right) ^{-1}=\begin{pmatrix}A^{\frac{1}{2}} & \boldsymbol{0}\\
\boldsymbol{0} & A^{\frac{1}{2}}
\end{pmatrix}^{-1}\begin{pmatrix}A^{\frac{1}{2}} & \boldsymbol{0}\\
\boldsymbol{0} & A^{\frac{1}{2}}
\end{pmatrix}^{-1}. \label{eq.inv2}
\]
  
Then we can rewrite the matrix in (\ref{A:InvCovDiff_def}) as
\begin{equation}
\begin{aligned}
2\left(\begin{array}{cc}
A & B\\
B^\intercal    & A
\end{array}\right) ^{-1}-\left(\begin{array}{cc}
A & \boldsymbol{0}\\
\boldsymbol{0} & A
\end{array}\right) ^{-1}=& \begin{pmatrix}A^{\frac{1}{2}} & \boldsymbol{0}\\
\boldsymbol{0} & A^{\frac{1}{2}}
\end{pmatrix}^{-1} \\
& \cdot \left\{2\begin{pmatrix}I & (A^{\frac{1}{2}}) ^{-1}B(A^{\frac{1}{2}}) ^{-1}\\
(A^{\frac{1}{2}}) ^{-1}B^\intercal   (A^{\frac{1}{2}}) ^{-1} & I
\end{pmatrix}^{-1}-I\right\} \\
& \cdot \begin{pmatrix}A^{\frac{1}{2}} & \boldsymbol{0}\\
\boldsymbol{0} & A^{\frac{1}{2}}
\end{pmatrix}^{-1}\label{eq:simga1-tilteSigma1-inverse-1}
\end{aligned}
\end{equation}

If $$ Q:=\left\{ 2\begin{pmatrix}I & (A^{\frac{1}{2}}) ^{-1}B(A^{\frac{1}{2}}) ^{-1}\\
(A^{\frac{1}{2}}) ^{-1}B^\intercal   (A^{\frac{1}{2}}) ^{-1} & I
\end{pmatrix}^{-1}-I\right\} $$ is positive definite, then $2\left(\begin{array}{cc}
A & B\\
B^\intercal    & A
\end{array}\right) ^{-1}-\left(\begin{array}{cc}
A & \boldsymbol{0}\\
\boldsymbol{0} & A
\end{array}\right) ^{-1}$ will be positive definite. 
\\Then the following is to show that matrix $Q$ is
positive definite. 
Denote $D:=(A^{\frac{1}{2}}) ^{-1}B^\intercal   (A^{\frac{1}{2}}) ^{-1}$.
Hence, 
\begin{align}\label{eq:inside Q}
\begin{pmatrix}I & (A^{\frac{1}{2}}) ^{-1}B(A^{\frac{1}{2}}) ^{-1}\\
(A^{\frac{1}{2}}) ^{-1}B^\intercal   (A^{\frac{1}{2}}) ^{-1} & I
\end{pmatrix}^{-1} & =\begin{pmatrix}I & D\\
D^\intercal    & I
\end{pmatrix}^{-1}\nonumber \\
\end{align}
Easy to tell $I-DD^\intercal$ and $I-D^{\intercal}D$  are positive definite. The inverse matrix of $\begin{pmatrix}I & D\\
D^\intercal    & I
\end{pmatrix}$  is $\begin{pmatrix}(I-DD^\intercal   ) ^{-1} & -D(I-D^\intercal   D) ^{-1}\\
-D^\intercal   (I-DD^\intercal   ) ^{-1} & (I-D^\intercal   D) ^{-1}
\end{pmatrix}\nonumber$. For convenience, $P:=I-D^\intercal   D\succ0$, $\Omega:=I-DD^\intercal   \succ0$. Observing that $P^{-1}=I+D^\intercal   \Omega^{-1}D$,
$\Omega^{-1}=I+DP^{-1}D^\intercal   $, $D^\intercal   \Omega^{-1}=P^{-1}D^\intercal   $.
Then we have 
\begin{equation}
\left(\begin{array}{cc}
I & D\\
D^{\intercal} & I
\end{array}\right)^{-1}=\left(\begin{array}{cc}
\Omega^{-1}& -DP^{-1}\\
-P^{-1}D^\intercal    & P^{-1}
\end{array}\right)\label{eq:IDID_inverse}
\end{equation}
Plug (\ref{eq:inside Q})(\ref{eq:IDID_inverse}) into $Q$, 
\begin{equation} \label{eq:Q}
\begin{aligned}
Q & =\begin{pmatrix}2\Omega^{-1}-I & -2DP^{-1}\\
-2P^{-1}D^\intercal    & 2P^{-1}-I
\end{pmatrix}.
\end{aligned}
\end{equation}
If we can prove the  Schur complement matrix of $Q$ is positive matrix, then so as $Q$. The Schur complement matrix of $Q$  (\ref{eq:Q}) is $W := \left(2\Omega^{-1}-I\right)-4DP^{-1}\left(2P^{-1}-I\right)^{-1}P^{-1}D^{\intercal}$. Directly computations yields $\left(2P^{-1}-I\right)^{-1}=P-PD^{\intercal}A^{-1}DP$,
where $A=I-DD^{\intercal}DD^{\intercal}$. Hence,
\begin{align*}
\left(2\Omega^{-1}-I\right)-4DP^{-1}\left(2P^{-1}-I\right)^{-1}P^{-1}D^{\intercal} & =\left(2\Omega^{-1}-I\right)-4DP^{-1}\left(P-PD^{\intercal}A^{-1}DP\right)P^{-1}D^{\intercal}\\
 & =\left(2\Omega^{-1}-I\right)-4\left(DP^{-1}D^{\intercal}-DD^{\intercal}A^{-1}DD^{\intercal}\right).
\end{align*}
Agian, by $D^{\intercal}\Omega^{-1}=P^{-1}D^{\intercal}$, we have
\begin{align}
\left(2\Omega^{-1}-I\right)-4\left(DP^{-1}D^{\intercal}-DD^{\intercal}A^{-1}DD^{\intercal}\right) & =\left(2\Omega^{-1}-I\right)-4\left(DD^{\intercal}\Omega^{-1}-DD^{\intercal}A^{-1}DD^\intercal\right).\label{eq:W}
\end{align}
Observing that $\Omega=I-DD^\intercal$ is positive definite, one concludes
that all eigenvalue of $DD^\intercal$ are in $(0,1)$. It follows that
if $\lambda$ is an eigenvalue of $DD^\intercal$, then

\[
\frac{2}{1-\lambda}-1-4\left(\frac{\lambda}{1-\lambda}-\frac{\lambda^{2}}{1-\lambda^{2}}\right)=\frac{1-\lambda}{1+\lambda}
\]
is an eigenvalue of $W$. As $\lambda\in\left(0,1\right)$,
$W$ is positive definite. Therefore $Q$ is positive definite, and
then we will have the Lemma. 
\end{proof}

\begin{lemma} \label{A:AB_pd}
Both $A,B\in\mathbb{R}^{n\times n}$, and if A is positive definite
matrix, B is positive (semi) definite matrix, then all eigenvalues
of AB or BA are positive (nonnegative).
\end{lemma}
\begin{proof}
$AB$ has the same eigenvalues as $BA$. $BA=A^{-\frac{1}{2}}(A^{\frac{1}{2}}BA^{\frac{1}{2}}) A^{\frac{1}{2}}$.
Then $BA$ has the same eigenvalues as $A^{\frac{1}{2}}BA^{\frac{1}{2}}$.
$\forall x\in\mathbb{R}^{n}$, $x^\intercal    A^{\frac{1}{2}}BA^{\frac{1}{2}}x=(A^{\frac{1}{2}}x) ^\intercal    B(A^{\frac{1}{2}}x) \succ0 (\succeq 0)$.
Then we will know $A^{\frac{1}{2}}BA^{\frac{1}{2}}$ is positive (semi) definite and
all its eigenvalues are positive (non-negative). Hence eigenvalues of $AB$
or $BA$ are positive (non-negative) too.
\end{proof}

\begin{remark} \label{A:singularValue}
$A,B\in\mathbb{R}^{n\times n}$, if $i+j-1\le n$, and the $i,j\in\mathbb{N}$,
then the singular value satisfies $\sigma_{i+j-1}(A+B) \leq \sigma_{i}(A) +\sigma_{j}(B) $,
where index $i$ is the $i^{th}$ largest singular value. If we set
$A=I_{p_{T}}-L_{p_{T}}$ , $B=L_{p_{T}}$, with $\Vert L_{p_{T}}\Vert<\mu<1$, then 
\[
\sigma_{p_{T}+1-1}(I_{p_{T}}-L_{p_{T}}+L_{p_{T}}) \leq \sigma_{p_{T}}(I_{p_{T}}-L_{p_{T}}) +\sigma_{1}(L_{p_{T}}) ,
\]
\[
1-\Vert L_{p_{T}}\Vert=1-\sigma_{1}(L_{p_{T}}) \leq \sigma_{p_{T}}(I_{p_{T}}-L_{p_{T}}) =\sigma_{\min}(I_{p_{T}}-L_{p_{T}}). 
\]
\end{remark}

\end{document}